\documentclass[11pt]{amsart} 
\usepackage{amsmath,amssymb,bm, color}

\newtheorem{theorem}{Theorem}[section]
\newtheorem{lemma}[theorem]{Lemma}

\newcommand{\bq}{\begin{equation}}
\newcommand{\eq}{\end{equation}}
\renewcommand{\ldots}{\dotsc}
\newtheorem{algorithm}{Modified Weak Galerkin Algorithm}

\newcommand{\bu}{{\bf u}}
\newcommand{\bw}{{\bf w}}

\newcommand{\bx}{{\bf x}}

\newcommand{\be}{{\bf e}}
\newcommand{\bv}{{\bf v}}
\newcommand{\bQ}{{\bf Q}}
\newcommand{\curl}{{\nabla\times}}
\newcommand{\cw}{{\nabla_w\times}}

\def\bpsi{{\boldsymbol{\psi}}}
\def\bvarphi{{\boldsymbol{\varphi}}}

\def\T{{\mathcal T}}
\def\E{{\mathcal E}}

\def\Q{{\mathbb Q}}
\def\pT{{\partial T}}
\def\l{{\langle}}
\def\r{{\rangle}}

\def\bbf{{\bf f}}

\def\bn{{\bf n}}
\def\bq{{\bf q}}

\def\3bar{{|\hspace{-.02in}|\hspace{-.02in}|}}
\def\p#1{\begin{pmatrix}#1\end{pmatrix}}  
\def\cal#1{\mathcal{#1}}
\def\ad#1{\begin{aligned}#1\end{aligned}}  \def\b#1{\mathbf{#1}} 
\def\a#1{\begin{align*}#1\end{align*}} \def\an#1{\begin{align}#1\end{align}}

\begin{document}

\title[MWG finite element]
{A Modified weak Galerkin finite element method for the Maxwell equations on polyhedral meshes}

\author {Chunmei Wang}
\address{Department of Mathematics, University of Florida, Gainesville, FL 32611, USA. }
\email{chunmei.wang@ufl.edu}
\thanks{The research of Chunmei Wang was partially supported by National Science Foundation Grants DMS-2136380 and DMS-2206332.}

\author {Xiu Ye}
\address{Department of Mathematics, University of Arkansas at
Little Rock, Little Rock, AR 72204, USA. }
\email{xxye@ualr.edu}

\author { Shangyou Zhang }
\address{Department of Mathematical
            Sciences, University
     of Delaware, Newark, DE 19716, USA. }
\email{szhang@udel.edu }

\date{}

\begin{abstract} 
We introduce a new numerical method for solving time-harmonic Maxwell's equations via
 the modified weak Galerkin technique.  The inter-element functions of the weak Galerkin
  finite elements are replaced by the average of the two discontinuous polynomial functions
  on the two sides of the polygon, in the modified weak Galerkin (MWG) finite element method.
With the dependent inter-element functions,
  the weak curl and the weak gradient are defined directly on totally discontinuous
  polynomials. Optimal-order convergence of the method is proved.   Numerical examples
  confirm the theory and show effectiveness of the modified weak Galerkin method over 
  the existing methods.

\end{abstract}



\maketitle
\baselineskip=14pt\parskip=5pt

\section{Introduction}

In this paper, we introduce a new numerical
method  for solving the time-harmonic Maxwell equations in a heterogeneous
medium $\Omega\subset \mathbb{R}^3$. The mixed formulation of this model problem seeks unknown
functions $\bu$ and $p$ satisfying
\begin{eqnarray}\label{model}
\nabla\times(\mu\nabla\times \bu)-\epsilon\nabla p &=&{\bf f}_1\quad \mbox{in}\;\Omega,\label{moment1}\\
\nabla\cdot(\epsilon\bu)&=&g_1\quad\mbox{in}\;\Omega,\label{cont1}\\
\bu\times\bn &=& 0\quad\mbox{on}\;\partial\Omega,\label{bcc1}\\
p&=&0\quad\mbox{on}\;\partial\Omega,\label{bc1}
\end{eqnarray}
where the constant coefficients $\mu>0$ and $\epsilon>0$ are the magnetic
permeability and the electric permittivity of the medium,
respectively.
The formulations (\ref{moment1})-(\ref{bc1}) of the Maxwell equations
  have been used in \cite{Mu-W-Y-Z,ps, psm,vd}, for better numerical stability.

The space $H(\hbox{curl};D)$ is defined as the set of vector-valued
functions on $D$ which, together with their curl, are square
integrable; i.e.,
\[
H({\rm curl}; D)=\left\{ \bv: \ \bv\in [L^2(D)]^3, \nabla\times\bv \in
[L^2(D)]^3\right\}.
\]
Denote
the subspace of 
$H({\rm curl}; D)$ with vanishing trace in the tangential component by
\[
H_0({\rm curl}; D)=\left\{ \bv: \ \bv\in [L^2(D)]^3, \nabla\times\bv \in
[L^2(D)]^3: \bv\times \bn|_{\partial D}=0 \right\}.
\]

A weak formulation for (\ref{moment1})-(\ref{bc1}) seeks $(\bu, p)
\in H_0(\hbox{curl};\Omega) \times H_0^1(\Omega)$ such that
\begin{eqnarray}\label{weakform}
(\nu\nabla\times\bu,\ \nabla\times\bv)-(\bv,\nabla p)&=&({\bf f},\  \bv),
\quad \forall \bv \in H_0(\hbox{curl};\Omega)\label{w1}\\
(\bu,\nabla q)&=&-(g,q),\quad\forall q\in H_0^1(\Omega),\label{w2}
\end{eqnarray}
where $\nu=\mu/\epsilon$, ${\bf f}={\bf f}_1/\epsilon$ and
$g=g_1/\epsilon$.

The numerical methods for the Maxwell equations have been studied extensively,
such as the $H(\hbox{curl};\Omega)$-conforming edge elements
\cite{boss,jin,monk, Nedelec, Nedelec1} and the discontinuous Galerkin
finite elements \cite{bls,bcnl, hps,hps-1,hps-2, ps,psm}.  
A weak Galerkin finite element method is studied for the Maxwell equations
 in \cite{Mu-W-Y-Z}.

In the weak Galerkin finite element method, the discrete finite element 
  functions are generalized functions denoted by 
$u_h=\{u_0, u_b\}$ where $u_0$ is a polynomial
  on each polyhedron and $u_b$ is an unrelated polynomial on each face-polygon.
Using the usual integration by parts, we can define weak derivatives of the  generalized
  function $\{u_0, u_b\}$ as the polynomial $L^2$-projections of the standard generalized 
   derivatives. 
The method was first introduced in
\cite{wy, wy-mixed} for second order elliptic equations, and 
applied to other partial differential equations 
\cite{cwang-jwang-biharmonic,sf-wg,cdg1,cdg2,cdg3,Ye-Zhang2020b,
   Ye-Zhang2020c,sf-hdiv,Ye-Zhang2021,Ye-Zhang2021S,Ye-Zhang2022,Ye-Zhang2022b,Ye-Zhang2023}.

In this work, similar to the weak Galerkin finite element method in \cite{Mu-W-Y-Z},
  the modified weak Galerkin finite element method,  which was proposed in \cite{XWang} and has been applied to various problems 
   \cite{Cui-Y-Z,Gao-Z-Z,Mu-XW-Y,XiuliWang},  is applied to the Maxwell equations.
In the modified weak Galerkin finite element method, the weak function
  on the inter-element polygons is replaced by the average on the two sides, i.e.,
    $u_h=\{u_0, u_b\}=\{u_0,\{u_0\}\}$.
  In particular, it is proved in \cite{Gao-Y-Z} that any convex or nonconvex combination
  of two-side functions would work in the modified weak Galerkin finite element
  method, i.e., $u_b=\theta u_1+(1-\theta) u_2$ for any real number $\theta$, instead of
   $\theta=1/2$ only previously.

In this paper, we prove that the modified weak Galerkin finite elements are
   inf-sup stable in solving the saddle-point problem \eqref{w1}--\eqref{w2}.
Consequently the optimal order of convergence for $\b u$ and $p$ is established.
Numerical examples are provided, showing the correctness of the theory and
  the efficiency over the existing weak Galerkin method.

This section is ended with some basic notations.
Let $D$ be any open bounded domain with Lipschitz continuous
boundary in $\mathbb{R}^3$. We use the standard definition for the
Sobolev space $H^s(D)$ and their associated inner products
$(\cdot,\cdot)_{s,D}$, norms $\|\cdot\|_{s,D}$, and seminorms
$|\cdot|_{s,D}$ for any $s\ge 0$. For example, for any integer $s\ge
0$, the seminorm $|\cdot|_{s, D}$ is given by
$$
|v|_{s, D} = \left( \sum_{|\alpha|=s} \int_D |\partial^\alpha v|^2
dD \right)^{\frac12}
$$
with the usual notation
$$
\alpha=(\alpha_1, \ldots, \alpha_d), \quad |\alpha| =
\alpha_1+\ldots+\alpha_d,\quad
\partial^\alpha =\prod_{j=1}^3\partial_{x_j}^{\alpha_j}.
$$
The Sobolev norm $\|\cdot\|_{s,D}$ is given by
$$
\|v\|_{s, D} = \left(\sum_{j=0}^s |v|^2_{j,D} \right)^{\frac12}.
$$

The space $H^0(D)$ coincides with $L^2(D)$, for which the norm and
the inner product are denoted by $\|\cdot \|_{D}$ and
$(\cdot,\cdot)_{D}$, respectively. When $D=\Omega$, we shall drop
the subscript $D$ in the norm and the inner product notation.

\section{MWG finite element scheme}\label{Section:WGFEM}

Let ${\cal T}_h$ be a partition of the domain $\Omega$ with mesh
size $h$ that consists of polyhedra of arbitrary shape. Assume that
the partition ${\cal T}_h$ is shape regular. Denote by ${\cal E}_h$ the set of all
faces in ${\cal T}_h$ and let ${\cal E}_h^0={\cal E}_h\backslash\partial\Omega$ be the set of all interior faces.

For simplicity, we adopt the following notations; i.e.,
\begin{eqnarray*}
(v,w)_{\T_h} &=& \sum_{T\in\T_h}(v,w)_T=\sum_{T\in\T_h}\int_T vw d\bx,\\
 \l v,w\r_{\partial\T_h}&=&\sum_{T\in\T_h} \l v,w\r_\pT=\sum_{T\in\T_h} \int_\pT vw ds.
\end{eqnarray*}
Let $P_k(K)$ consist all the polynomials of degree less or equal to $k$ defined on $K$.

Let $k\ge 1$. We define two finite element spaces $V_h$ and $W_h$ as follows
\begin{eqnarray*}
V_h &=&\left\{ \bv\in [L^2(\Omega)]^3:\ \bv|_{T}\in [P_{k}(T)]^3,\;\; T\in\T_h\right\},\\
W_h &=&\left\{ q\in L^2(\Omega):\ q|_{T}\in P_{k-1}(T),\;\; T\in\T_h\right\}.
\end{eqnarray*}

We further introduce two subspaces of $V_h$ and $W_h$ with boundary conditions
\begin{eqnarray*}
V_h^0 &=&\left\{ \bv\in V_h:\ \bv\times\bn|_{\partial\Omega}=0\right\},\\
W_h^0 &=&\left\{ q\in W_h:\ q|_{\partial\Omega}=0 \right\}.
\end{eqnarray*}

For $T\in\T_h$, we define the jump of $\tau$ as $[\tau]_\pT=\frac12(\tau|_T-\tau|_{T_n})$ and the average of $\tau$ as $\{\tau\}_\pT=\frac12(\tau|_T+\tau|_{T_n})$, where $T_n$ denotes the elements neighboring to $T$.  For $e\in\pT\cap\partial\Omega$, we define $[\tau]_e=\tau$ and $\{\tau\}_e=0$. Then we have the following identities
\begin{eqnarray}
\langle \bq\times\bn,\; \bv\rangle_{\partial\T_h}&=&\langle [\bq]\times\bn,\;\bv\rangle_{\partial\T_h}
+\langle \{\bq\}\times\bn,\;\bv\rangle_{\partial\T_h},\label{key0}\\
\langle \bv\cdot\bn,\; q\rangle_{\partial\T_h}&=&\langle \bv\cdot\bn,\; [q]\rangle_{\partial\T_h}
+\langle \bv\cdot\bn,\;\{q\}\rangle_{\partial\T_h}.\label{key00}
\end{eqnarray}

For a function $\bv\in V_h$, its  weak curl $\nabla_w\times\bv$ is defined as a piecewise vector-valued polynomial such that on each $T\in\T_h$, $\nabla_w\times\bv \in [P_{k-1}(T)]^3$   satisfies  
\begin{equation}\label{d-c}
  (\nabla_w\times\bv, \bvarphi)_T = (\bv, \nabla\times\bvarphi)_T
 -\langle \{\bv\}\times\bn, \bvarphi\rangle_{\partial T},\qquad
   \forall \bvarphi\in [P_{k-1}(T)]^3.
\end{equation}

For a function $q\in W_h$, its  weak gradient $\nabla_w q$  is defined as a piecewise vector-valued polynomial such that on each $T\in\T_h$, $\nabla_w q \in [P_k(T)]^3$  satisfies  
\begin{equation}\label{d-g}
  (\nabla_w q, \bvarphi)_T = -(q, \nabla\cdot\bvarphi)_T
 + \langle \{ q\}, \bvarphi\cdot\bn\rangle_{\partial T},\qquad
   \forall \bvarphi\in [P_k(T)]^3.
\end{equation}

We define a bilinear form with an appropriate
stabilization term as follows:
\begin{eqnarray}\label{Stabilized-A-form}
a(\bv,\ \bw)&=&(\nu\nabla_w\times\bv,\
\nabla_w\times\bw)_{{\cal T}_h}+s_1(\bv,\bw),
\end{eqnarray}
where
\begin{eqnarray}\label{Stabilization-T1}
s_1(\bv,\;\bw) &=& \sum_{T\in {\cal T}_h}h^{-1}_T (\l [\bv]\times\bn,\;\;[\bw]\times\bn\r_\pT
   \\ & & \quad \nonumber +\l
[\bv]\cdot\bn,\;\;[\bw]\cdot\bn\r_{\pT\setminus {\partial\Omega}}).
\end{eqnarray}
We introduce the following bilinear form
\begin{eqnarray}\label{B-form}
b(\bv,\ q)&=&(\bv,\nabla_w q)_{{\cal T}_h},
\end{eqnarray}
and a second stabilization term
\begin{eqnarray}\label{Stabilization-T2}
s_2(p,\;q) & = & \sum_{T\in {\cal T}_h}h_T\l
[p],\;\;[q]\r_\pT.
\end{eqnarray}

We are ready to present the MWG method  for the weak formulation \eqref{weakform} of the time-harmonic Maxwell model equations \eqref{model}.
\begin{algorithm}
Find $\bu_h\in V_h^0$  and
$p_h\in W_h^0$ satisfying
\begin{eqnarray}
a(\bu_h,\ \bv)-b(\bv,\;p_h)&=&({\bf f},\;\bv)\quad\forall\ \bv\in
V_{h}^0,\label{wg1}\\
b(\bu_h,\;q)+s_2(p_h,q)&=&-(g,q) \quad\forall\ q\in W_h^0.\label{wg2}
\end{eqnarray}
\end{algorithm}
\smallskip

\begin{lemma}\label{lemma-ue}
The modified weak Galerkin finite element algorithm (\ref{wg1})-(\ref{wg2})
has a unique solution.
\end{lemma}

\smallskip

\begin{proof}
It suffices to show that zero is the only solution of
(\ref{wg1})-(\ref{wg2}) if $\bbf=0$ and $g=0$.
 Letting $\bv=\bu_h$
and $q=p_h$ in (\ref{wg1})-(\ref{wg2}) 
gives
\begin{eqnarray*}
 & & (\nu\cw\bu_h,\ \cw\bu_h)_{\T_h}  + \sum_{T\in {\cal T}_h}h^{-1}_T (\l [\bu_h]\times\bn,\;\;[\bu_h]\times\bn\r_\pT \\
  & &\quad +\l
[\bu_h]\cdot\bn,\;\;[\bu_h]\cdot\bn\r_{\pT\setminus {\partial\Omega}}) 
    + \sum_{T\in\T_h}h_T\l [p_h],\ [p_h]\r_\pT=0,
\end{eqnarray*}
which implies $\nabla_w\times\bu_h=0$ on each $T$, $\bu_h$ is continuous in its tangential and normal directions at element boundary and
 $[p_h]=0$ on $\pT$.
Then, it follows from (\ref{d-c}) and the usual 
integration by parts that for any $\bv\in V_h^0$,
\begin{eqnarray*}
0&=&(\cw\bu_h,\bv)\\
&=&(\bu_h,\ \curl\bv)_{\T_h}-\l\{\bu_h\}\times\bn,\ \bv\r_{\partial\T_h}\\
&=&(\curl\bu_h,\ \bv)_{\T_h}+\l\bu_h\times\bn,\ \bv\r_{\pT_h}-\l\{\bu_h\}\times\bn,\ \bv\r_{\partial\T_h}\\
&=&(\curl\bu_h,\ \bv)_{\T_h}+\l [\bu_h]\times\bn,\ \bv\r_{\partial\T_h}\\
&=&(\curl\bu_h,\ \bv)_{\T_h}+\l  \bu_h \times\bn,\ \bv\r_{\partial\Omega}\\
&=&(\curl\bu_h,\ \bv)_{\T_h},
\end{eqnarray*}
which gives $\curl\bu_h=0$ on each $T\in\T_h$ due to $\bu_h\in V_h^0$, which, together with $\bu_h$ is continuous in its tangential and normal directions at element boundary, yields $\curl\bu_h=0$ in the domain $\Omega$.  Using  Recall that  $[p_h]=0$ on $\pT$. It follows from (\ref{wg2}), \eqref{d-g} and the facts $[\bu_h]\cdot\bn=0$ on $e\in\E_h^0$ and $\{q\}=0$ on $\partial\Omega$,
\begin{eqnarray*}
0&=&(\bu_h,\nabla_w q)_{\T_h}=-(\nabla\cdot\bu_h, q)_{\T_h}+\l\bu_h\cdot\bn,\ \{q\}\r_{\partial\T_h}\\
&=&-(\nabla\cdot\bu_h, q)_{\T_h}.
\end{eqnarray*}
Letting $q=\nabla\cdot\bu_h$ in the above equation
yields $\nabla\cdot \bu_h=0$ on each $T\in\T_h$.

Note that $\nabla\times \bu_h=0$ in the domain $\Omega$. Thus, there exists a potential
function $\phi$ such that $\bu_h=\nabla\phi$ in $\Omega$. It follows
from $\nabla\cdot \bu_h=0$ and the fact
that $\bu_h\cdot\bn$ is continuous
that $\Delta\phi=0$ is strongly satisfied in $\Omega$. The boundary
condition of (\ref{bcc1}) implies that
$\bu_h\times\bn=\nabla\phi\times\bn=0$ on $\partial\Omega$.
Therefore, $\phi$ must be a constant on $\partial\Omega$. The
uniqueness of the solution of the Laplace equation implies that
$\phi=const$ is the only solution of $\Delta\phi=0$ if $\Omega$ is
simply connected. Then we must have $\bu_h=\nabla\phi=0$.

Since $\bu_h=0$, we then have $b(\bv,p_h)=0$ for any $\bv\in
V_{h}^0$ by (\ref{wg1}). It follows from the definition of $b(\cdot,\cdot)$ and
\eqref{d-g} that
\begin{eqnarray}\label{eu1}
0&=&b(\bv,\ p_h)=(\bv,\nabla_w
p_h)_{\T_h}\\
&=&-(\nabla\cdot\bv,\ p_h)_{\T_h}+\l \bv\cdot\bn,\ \{p_h\}\r_{\partial\T_h}\nonumber\\
&=&(\bv,\nabla p_h)_{\T_h}-\l \bv\cdot\bn,\ [p_h]\r_{\partial\T_h},\nonumber\\
&=&(\bv,\nabla p_h)_{\T_h},\nonumber
\end{eqnarray}
where we have used the fact that $[p_h]=0$ on $\partial T$. Letting
$\bv=\nabla p_h$ in (\ref{eu1}) gives $\nabla p_h=0$ on each $T\in \T_h$, i.e. $p_h$ is a constant on $T\in\T_h$. Using the fact that  $p_h=[p_h]=0$
 on $\partial\Omega$, we obtain $p_h=0$ in $\Omega$. 
 
 This completes the proof of the Lemma.
 
\end{proof}

\section{Error Equations}\label{Section:ErrorEquation}

Denote by $\bQ_k$ \ and \ $Q_{k-1}$ the element-wise defined $L^2$
projections onto $[P_k(T)]^3$ and $P_{k-1}(T)$ for each element $T\in \T_h$, respectively.
We define
two error functions 
$$
\be_h=\bQ_k\bu -\bu_h, \qquad \epsilon_h=Q_{k-1}p-p_h.
$$ 
Next we will derive the equations that $\be_h$ and $\epsilon_h$ satisfy.
For simplicity of analysis, we assume that the coefficient $\nu$ in (\ref{w1}) is a piecewise constant
function with respect to the finite element partition $\T_h$.

\begin{lemma} For $\bv,\bw\in V_h$, we have
\begin{equation}\label{key}
(\cw \bv, \bw)_{\T_h} = ( \curl\bv, \bw)_{\T_h}+ \l [\bv]\times\bn, \bw\r_{\partial\T_h}.
\end{equation}
\end{lemma}
\begin{proof}
Using (\ref{d-c}), the usual integration by parts, and (\ref{key0}), we have
\begin{eqnarray*}
(\cw \bv,\; \bw)_{\T_h} &=& (\bv,\; \curl\bw)_{\T_h} -  \langle
\{\bv\}\times\bn,\; \bw\rangle_{\partial\T_h}\\
&=& (\curl\bv,\; \bw)_{\T_h} + \langle (\bv-\{\bv\})\times\bn,\; \bw\rangle_{\partial\T_h}\\
&=&(\curl\bv,\; \bw)_{\T_h}+\l [\bv]\times\bn, \bw\r_{\partial\T_h}.
\end{eqnarray*}
This completes the lemma.
\end{proof}

\begin{lemma}
Let $(\bu_h, p_h)\in V_h\times W_h$ be the MWG finite element solution arising from
(\ref{wg1}) and (\ref{wg2}). For any $\bv\in V_{h}$ and $q\in W_h$, the following error equations hold true:
\an{
a(\be_h,\ \bv)-b(\bv,\ \epsilon_h)&= -E_1(\bu,\bv)-E_2(\bu,\bv)+E_3(p,\bv)\label{ee1}\\
             &\quad \ +s_1(\bQ_k\bu,\bv), \nonumber \\
b(\be_h,\ q)+s_2(\epsilon_h,\ q)&= E_4(\bu,q)+s_2(Q_{k-1}p,q),\label{ee2} }
where
\begin{eqnarray*}
E_1(\bu,\bv)&=&\l\nu\curl\bu-\bQ_k(\nu\curl\bu),\,[\bv]\times\bn\r_{\partial\T_h},\\
E_2(\bu,\bv)&=&(\bQ_k\nu\curl\bu-\nu\cw\bQ_k\bu,\ \cw\bv)_{\T_h},\\
E_3(p,\bv)&=& \l \{p-Q_{k-1}p\},\,[\bv]\cdot\bn\r_{\partial\T_h},\\
E_4(\bu,q)&=&\l (\bu-\bQ_k\bu)\cdot\bn,\,[q]\r_{\partial\T_h}.
\end{eqnarray*}
\end{lemma}

\begin{proof}
Testing (\ref{moment1}) by $\bv\in V_h$ gives
\begin{equation}\label{m1}
(\curl(\nu\curl\bu),\;\bv)- (\nabla p,\ \bv)=(\bbf,\; \bv).
\end{equation}
Using the usual integration by parts, (\ref{d-c}) and (\ref{key0}), we have
\begin{eqnarray}
 & & (\curl(\nu\curl\bu),\;\bv)_{\T_h}\nonumber\\
 &=&(\nu\curl\bu,\ \curl\bv)_{\T_h}+\l\nu\curl\bu,\,\bv\times\bn\r_{\partial\T_h}\nonumber\\
&=&(\bQ_k(\nu\curl\bu),\ \curl\bv)_{\T_h}+\l\nu\curl\bu,\,[\bv]\times\bn\r_{\partial\T_h}\nonumber\\
&=&(\curl\bQ_k(\nu\curl\bu),\ \bv)_{\T_h}-\l\bQ_k(\nu\curl\bu),\,\bv\times\bn\r_{\partial\T_h}\nonumber\\
& &\quad +\l\nu\curl\bu,\,[\bv]\times\bn\r_{\partial\T_h}\nonumber\\
&=&(\nu\curl\bu,\ \cw\bv)_{\T_h}-\l\bQ_k(\nu\curl\bu),\,[\bv]\times\bn\r_{\partial\T_h}\nonumber\\
& &\quad +\l\nu\curl\bu,\,[\bv]\times\bn\r_{\partial\T_h}\nonumber\\
&=&(\nu\curl\bu,\ \cw\bv)_{\T_h}+\l\nu\curl\bu-\bQ_k(\nu\curl\bu),\,
    [\bv]\times\bn\r_{\partial\T_h}\nonumber\\
&=&(\bQ_k\nu\curl\bu,\ \cw\bv)_{\T_h}+E_1(\bu,\bv)\nonumber\\
&=&(\bQ_k\nu\curl\bu,\ \cw\bv)_{\T_h}+(\nu\cw\bQ_k\bu,\ \cw\bv)_{\T_h}\nonumber\\
& &\quad -(\nu\cw\bQ_k\bu,\ \cw\bv)_{\T_h}+E_1(\bu,\bv)\nonumber\\
&=&(\nu\cw\bQ_k\bu,\ \cw\bv)_{\T_h}+E_2(\bu,\bv)+E_1(\bu,\bv).\label{m7}
\end{eqnarray}
Using the usual integration by parts,  (\ref{d-g}) and (\ref{key00}) gives
\begin{eqnarray}
(\nabla p,\ \bv)&=&-(p,\ \nabla\cdot\bv)_{\T_h}+\l  p,\,\bv\cdot\bn\r_{\partial\T_h}\nonumber\\
&=&-(Q_{k-1} p,\ \nabla\cdot\bv)_{\T_h}+\l p,\,\bv\cdot\bn\r_{\partial\T_h}\nonumber\\
&=&(\nabla_w Q_{k-1} p,\ \bv)_{\T_h}+\l \{p-Q_{k-1}p\},\,\bv\cdot\bn\r_{\partial\T_h}\nonumber\\
&=&(\nabla_w Q_{k-1} p,\ \bv)_{\T_h}+\l \{p-Q_{k-1}p\},\,[\bv]\cdot\bn\r_{\partial\T_h}\nonumber\\
&=&(\nabla_w Q_{k-1} p,\ \bv)_{\T_h}+E_3(p,\bv).\label{m8}
\end{eqnarray}
Combining the two equations (\ref{m7}) and (\ref{m8}) with (\ref{m1}) gives
\a{ &\quad \
(\nu\cw\bQ_k\bu,\ \cw\bv) -(\nabla_w Q_{k-1} p,\ \bv)_T\\
   & =(\bbf,\; \bv)-E_1(\bu,\bv)-E_2(\bu,\bv)+E_3(p,\bv). }
Further we add $s_1(\bQ_k\bu,\bv)$ to the both sides of the above equation  to obtain
\an{ \label{m4}\ad{
a(\bQ_k\bu,\bv)- b(\bv, Q_{k-1} p)&=(\bbf,\; \bv)-E_1(\bu,\bv)\\
                &\quad \ -E_2(\bu,\bv)+E_3(p,\bv)+s_1(\bQ_k\bu,\bv). } }
Subtracting (\ref{wg1}) from (\ref{m4})   yields (\ref{ee1}).

Testing (\ref{cont1}) by $q\in W_h$ gives
\begin{equation}\label{m5}
(\nabla\cdot\bu,\;q)=(g,\; q).
\end{equation}
It follows from the usual integration by parts, (\ref{d-g}) and (\ref{key00})
\begin{eqnarray*}
(\nabla\cdot\bu,\ q)_{\T_h}&=&-(\bu,\ \nabla q)_{\T_h}+\l \bu\cdot\bn,\,q\r_{\partial\T_h}\\
&=&-(\bQ_k \bu,\ \nabla q)_{\T_h}+\l \bu\cdot\bn,\,q\r_{\partial\T_h}\\
&=&(\nabla\cdot\bQ_k\bu,\  q)_{\T_h}+\l (\bu-\bQ_k\bu)\cdot\bn,\,q\r_{\partial\T_h}\\
&=&(\nabla\cdot\bQ_k\bu,\  q)_{\T_h}-\l \bQ_k\bu\cdot\bn,\,\{q\}\r_{\partial\T_h}\\
  && \quad 
 +\l \bQ_k\bu\cdot\bn,\,\{q\}\r_{\partial\T_h}+\l (\bu-\bQ_k\bu)\cdot\bn,\,q\r_{\partial\T_h}\\
&=&-(\bQ_k\bu,\  \nabla_w q)_{\T_h}-\l (\bu-\bQ_k\bu)\cdot\bn,\,\{q\}\r_{\partial\T_h}
  \\ &&\quad +\l (\bu-\bQ_k\bu)\cdot\bn,\,q\r_{\partial\T_h}\\
&=&-(\bQ_k\bu,\  \nabla_w q)_{\T_h}+\l (\bu-\bQ_k\bu)\cdot\bn,\,[q]\r_{\partial\T_h}\\
&=&-(\bQ_k\bu,\  \nabla_w q)_{\T_h}+E_4(\bu,q).
\end{eqnarray*}
Using the equation above, (\ref{m5}) becomes
\[
b(\bQ_k\bu,q)=-(g,\; q)+E_4(\bu,q).
\]
Adding  $s_2(Q_{k-1}p,q)$  to the both sides of the equation above gives
\begin{equation}\label{m6}
b(\bQ_k\bu,q)+s_2(Q_{k-1}p,q)=-(g,\; q)+E_4(\bu,q)+s_2(Q_{k-1}p,q).
\end{equation}
Subtracting  (\ref{wg2}) from (\ref{m6}) gives (\ref{ee2}).

This completes the proof of the lemma.
\end{proof}

\section{Preparation for Error Estimates}\label{Section:L2Projections}

For $\bv\in V_{h}$, we define a semi-norm
\begin{equation}\label{3barnorm}
\3bar \bv\3bar^2=a(\bv,\;\bv)=\sum_{T\in\T_h}\nu\|\cw\bv\|_T^2
   +\sum_{T\in\T_h}h_T^{-1}\|[\bv]\|_\pT^2.
\end{equation}
For $q\in W_h$,  we define another semi-norm  
\begin{equation}\label{qnorm}
|q|_{0,h}^2=\sum_{T\in\T_h}h_T\|[q]\|_\pT^2.
\end{equation}

Let $T$ be an element with $e$ as a face.  For any function $\phi\in
H^1(T)$, the following trace inequality has been proved for
arbitrary polyhedra  $T$ in \cite{wy-mixed}; i.e.,
\begin{equation}\label{trace}
\|\phi\|_{e}^2 \leq C \left( h_T^{-1} \|\phi\|_T^2 + h_T \|\nabla
\phi\|_{T}^2\right).
\end{equation}

If $\phi$ is a polynomial on the element $T\in{\cal T}_h$, we have
from the inverse inequality that
\begin{equation}\label{trace2}
\|\phi\|_{e}^2 \leq Ch_T^{-1}\|\phi\|_T^2.
\end{equation}

\begin{lemma}\label{Lemma:myestimates}
Let $(\bw,p)\in [H^{t+1}(\Omega)]^3\times H^t(\Omega)$ with $\bw\times\bn=0$ and $p=0$ on $\partial\Omega$ and $(\bv,q)\in
V_h\times W_h$ with $\frac12 <t\le k$. Then
\begin{eqnarray}
|s_1(\bQ_k\bw,\ \bv)| &\le& Ch^t\|\bw\|_{t+1} \3bar\bv\3bar,\label{s--1}\\
|s_2(Q_{k-1} p,\ q)|&\le& Ch^t\|p\|_{t}\ |q|_{0,h},\label{s2}\\
|E_1(\bw,\bv)|&\le& Ch^t\|\bw\|_{t+1} \3bar\bv\3bar,\label{E1}\\
|E_2(\bw,\bv)|&\le& Ch^t\|\bw\|_{t+1} \3bar\bv\3bar,\label{E2}\\
|E_3(p,\bv)|&\le& Ch^t\|p\|_{t}\3bar\bv\3bar, \label{E3}\\
|E_4(\bw,q)|&\le& Ch^t\|\bw\|_{t+1}|q|_{0,h}.\label{E4}
\end{eqnarray}
\end{lemma}

\begin{proof}
Using the trace inequality (\ref{trace}), the Cauchy-Schwarz inequality, and the properties of the projection operators $\bQ_k$ and $Q_{k-1}$, we have
\begin{eqnarray*} & &|s_1(\bQ_k\bw,\ \bv)|\\
&=&|\sum_{T\in\T_h} h_T^{-1}(\langle [\bQ_k\bw]\times\bn,\; [\bv]\times\bn\rangle_\pT+\langle [\bQ_k\bw]\cdot\bn,\; [\bv]\cdot\bn\rangle_{\pT\setminus\partial\Omega})|\\
&=&|\sum_{T\in\T_h} h_T^{-1}(\langle [\bQ_k\bw-\bw]\times\bn,\; [\bv]\times\bn\rangle_\pT\\
 &&\quad +\langle [\bQ_k\bw-\bw]\cdot\bn,\; [\bv]\cdot\bn\rangle_{\pT\setminus\partial\Omega})|\\
&\le& C\left(\sum_{T\in\T_h}(h_T^{-2}\|\bQ_k\bw-\bw\|_T^2+\|\nabla (\bQ_k\bw-\bw)\|_T^2)\right)^{1/2} \3bar\bv\3bar\\
&\le& Ch^t\|\bw\|_{t+1} \3bar\bv\3bar,
\end{eqnarray*}
and
\begin{eqnarray*}
|s_2(Q_{k-1} p,\ q)|&=&|\sum_{T\in\T_h} h_T\langle [Q_{k-1} p],\; [q]\rangle_\pT|\\
&\le &|\sum_{T\in\T_h}h_T |\langle [Q_{k-1}p- p],\; [q]\rangle_\pT|\\
&\le& Ch^t\|p\|_{t}\ |q|_{0,h}.
\end{eqnarray*}

Similarly, we have  that
\begin{eqnarray*}
|E_1(\bw,\ \bv)|&=& |\l\nu\curl\bw-\bQ_k(\nu\curl\bw),\,[\bv]\times\bn\r_{\partial\T_h}|\\
&\le& (\sum_{T\in\T_h}h_T\|(I-\bQ_k)\curl\bw\|_\pT^2)^{1/2}\3bar\bv\3bar\\
&\le& Ch^t\|\bw\|_{t+1} \3bar\bv\3bar,
\end{eqnarray*}
and
\begin{eqnarray*}
|E_3(p,\bv)|&=& |\l \{p-Q_{k-1}p\},\,[\bv]\cdot\bn\r_{\partial\T_h}|\\
&\le& (\sum_{T\in\T_h}h_T\|p-Q_{k-1}p\|_\pT^2)^{1/2}\3bar\bv\3bar\\
&\le& Ch^t\|p\|_{t} \3bar\bv\3bar,
\end{eqnarray*}
and
\begin{eqnarray*}
|E_4(\bw,q)|&=&|\l (\bw-\bQ_k\bw)\cdot\bn,\,[q]\r_{\partial\T_h}|\\
&\le& (\sum_{T\in\T_h}h_T^{-1}\|\bw-\bQ_k\bw\|_\pT^2)^{1/2} |q|_{0,h}\\
&\le& Ch^t\|\bw\|_{t+1} |q|_{0,h}.
\end{eqnarray*}

In order to bound $E_2(\bw,\bv)$, for $\bv\in V_h$, it follows from (\ref{d-c}) and (\ref{key0}) that
\an{\label{e10}\ad{
 &\quad \ (\bQ_k\curl\bw-\cw\bQ_k\bw,\ \bv)_{\T_h}\\
&=  (\bQ_k\curl\bw,\ \bv)_{\T_h}-(\cw\bQ_k\bw,\ \bv)_{\T_h}\\
&=  (\bw,\ \curl\bv)_{\T_h}-\l\bw\times\bn, \bv\r_{\partial\T_h}\\
   &\quad\ -(\bQ_k\bw,\ \curl\bv)_{\T_h}-\l \{\bQ_k\bw\}\times\bn, \bv\r_{\partial\T_h}\\
&=  (\bw,\ \curl\bv)_{\T_h}-\l\bw\times\bn, [\bv]\r_{\partial\T_h}-(\bQ_k\bw,\ \curl\bv)_{\T_h}
  \\&\quad \ -\l \{\bQ_k\bw\}\times\bn, [\bv]\r_{\partial\T_h}\\
&= (\bw-\bQ_k\bw,\ \curl\bv)_{\T_h}+\l \{\bw-\bQ_k\bw\}, [\bv]\times\bn\r_{\partial\T_h}\\
&= \l \{\bw-\bQ_k\bw\}, [\bv]\times\bn\r_{\partial\T_h}.  } }
Letting $\bv=\bQ_k\curl\bw-\cw\bQ_k\bw$ in \eqref{e10} and using the trace inequalities \eqref{trace}-\eqref{trace2} gives
\begin{eqnarray}
\|\bQ_k\curl\bw-\cw\bQ_k\bw\|\le Ch^t\|\bw\|_{t+1}.\label{p1}
\end{eqnarray}
Using Cauchy-Schwarz inequality and  the estimate (\ref{p1}) implies
\begin{eqnarray*}
|E_2(\bw,\bv)|&=&|(\bQ_k(\nu(\curl\bw)-\nu\cw\bQ_k\bw),\ \cw\bv)_{\T_h}|\\
&\le& Ch^t\|\bw\|_{t+1} \3bar\bv\3bar.
\end{eqnarray*}

This completes the proof of the lemma.
\end{proof}

\section{Error Estimates}\label{Section:H1ErrorEstimates}

The objective of this section is to establish the optimal order
error estimates for $\bu_h$ and $p_h$ in certain discrete norms.

\begin{lemma}\label{Lemma:inf-sup}
For any $q\in W_h$, there exist a $\bv_q\in V_{h}$ with $\bv|_T=h_T^2\nabla q$ such that
\begin{equation}\label{inf-sup}
b(\bv_q,q)\ge \sum_{T\in\T_h} h_T^2 \|\nabla q\|_{T}^2-C \sum_{T\in\T_h} h_T\|[q]\|_\pT^2
\end{equation}
and
\begin{equation}\label{inf-sup-boundedness}
\3bar \bv_q\3bar^2 \leq C \sum_{T\in\T_h} h_T^2 \|\nabla q\|_T^2,
\end{equation}
where $C$ is a constant independent of $h$.
\end{lemma}

\begin{proof}
For a given $q\in W_h$ and $\bv\in V_{h}$, it follows from the usual integration by parts, (\ref{key00}) and \eqref{d-g} that
\begin{equation*}
\begin{split}
b(\bv,q) &= (\bv, \nabla_w q)_{\T_h} \\
&= \l\bv\cdot\bn, \{q\} \r_{\partial\T_h}-
(\nabla\cdot\bv, q)_{\T_h}\\
&=  (\bv, \nabla q)_{\T_h} -\l\bv\cdot\bn,
[q]\r_{\partial\T_h}.\end{split}
\end{equation*}
By choosing $\bv|_T=2h_T^2\nabla q$, we arrive at
\begin{equation*}
b(\bv,q) = 2\sum_{T\in\T_h} h_T^2(\nabla q, \nabla q)_{T}
-\sum_{T\in\T_h} 2h_T^2\l\nabla q\cdot\bn, [q]\r_{\pT}.
\end{equation*}
Now by the Cauchy-Schwarz inequality, the trace inequality
(\ref{trace2}) and the inverse inequality we obtain
\begin{equation*}
\begin{split}
b(\bv,q) & \ge  2\sum_{T\in\T_h} h_T^2(\nabla q, \nabla q)_{T}
-2\sum_{T\in\T_h} h_T^2\|\nabla q\cdot\bn\|_{\pT} \|[q]\|_{\pT}\\
&\ge \sum_{T\in\T_h}2h_T^2(\nabla q, \nabla q)_T - \sum_{T\in\T_h}Ch_T^{1.5}\|\nabla q\|_{T} \|[q]\|_{\pT}\\
&\ge \sum_{T\in\T_h}h_T^2\|\nabla q\|_T^2 - C\sum_{T\in\T_h}h_T \|[q]\|_{\pT}^2,
\end{split}
\end{equation*}
which gives rise to the inequality (\ref{inf-sup}). The boundedness
estimate (\ref{inf-sup-boundedness}) can be obtained by computing the
triple bar norm of $\bv_q$ directly with the trace and the inverse inequalities. This completes the proof of the
lemma.
\end{proof}

\begin{theorem}\label{h1-bd}
Let $(\bu, p)\in [H^{t+1}(\Omega)]^3\times [H^{t}(\Omega)]$ with $\frac12< t\le k$ and $(\bu_h,p_h)\in
V_h\times W_h$ be the solution of (\ref{moment1})-(\ref{bc1}) and
(\ref{wg1})-(\ref{wg2}) respectively. There holds
\begin{eqnarray}
\3bar\be_h\3bar+|\epsilon_h|_{0,h}&\le&
Ch^t(\|\bu\|_{t+1}+\|p\|_t),\label{err1}\\
(\sum_{T\in\T_h} h_T^2 \|\nabla\epsilon_h\|_T^2)^\frac12 &\leq&Ch^t(\|\bu\|_{t+1}+\|p\|_t).
\label{err1-secondpart}
\end{eqnarray}
\end{theorem}

\smallskip

\begin{proof}
By letting $\bv=\be_h$ in (\ref{ee1}) and $q=\epsilon_h$ in
(\ref{ee2}) and adding the two resulting equations, we have
\an{ \ad{
\3bar\be_h\3bar^2+|\epsilon_h|_{0,h}^2&= -E_1(\bu,\be_h)-E_2(\bu,\be_h)+E_3(p,\be_h)\\
&\quad \ +s_1(\bQ_k\bu,\be_h)+E_4(\bu,\epsilon_h)+s_2(Q_{k-1} p, \epsilon_h). }
\label{main} }
Substituting the estimates (\ref{s--1})-(\ref{E4}) into (\ref{main}) yields
\begin{eqnarray}
\3bar\be_h\3bar^2+|\epsilon_h|_{0,h}^2&\le&Ch^t(\|\bu\|_{t+1}+\|p\|_t)(\3bar \be_h\3bar+ |\epsilon_h|_{0,h})\label{b-u}
\end{eqnarray}
which implies the error estimate (\ref{err1}).

It follows from (\ref{ee1}) with $\bv|_T=\bv_{\epsilon_h}|_T=h_T^2\nabla \epsilon_h$ that
\a{
b(\bv_{\epsilon_h},\ \epsilon_h)=a(\be_h,\ & \bv_{\epsilon_h})+E_1(\bu,\bv_{\epsilon_h})+E_2(\bu,\bv_{\epsilon_h})\\ & -E_3(p,\bv_{\epsilon_h})-s_1(\bQ_k\bu,\bv_{\epsilon_h}),
 }
 which, together with Cauchy-Schwarz inequality, (\ref{inf-sup}), and (\ref{s--1})-(\ref{E4}), yields
\begin{equation}\label{raining.100}
\begin{split} &\quad \ 
\sum_{T\in\T_h} h_T^2\|\nabla\epsilon_h\|_{T}^2\\
&\leq |b(\bv_{\epsilon_h},
\epsilon_h)| + C |\epsilon_h|_{0,h}^2\\
&\leq |a(\be_h,\ \bv_{\epsilon_h})|+|E_1(\bu,\bv_{\epsilon_h})|+|E_2(\bu,\bv_{\epsilon_h})|\\
&+|E_3(p,\bv_{\epsilon_h})|+|s_1(\bQ_k\bu,\bv_{\epsilon_h})|+C |\epsilon_h|_{0,h}^2\\
&\leq \3bar \be_h\3bar\ \3bar \bv_{\epsilon_h}\3bar +
h^t(\|\bu\|_{t+1}+\|p\|_t) \3bar\bv_{\epsilon_h}\3bar + C
 |\epsilon_h|_{0,h}^2\\
&\leq C\Big(\3bar \be_h\3bar + h^t(\|\bu\|_{t+1}+\|p\|_t)\Big)
(\sum_{T\in\T_h} h_T^2 \|\nabla\epsilon_h\|_T^2)^\frac12 + C |\epsilon_h|_{0,h}^2,
\end{split}
\end{equation}
where we have used the estimate (\ref{inf-sup-boundedness})  in the
last inequality. The estimate (\ref{err1-secondpart}) can be obtained  from (\ref{raining.100}) and
(\ref{err1}). 

This completes the proof of the theorem.
\end{proof}

\section{An Error Estimate in $L^2$ norm}\label{Section:L2ErrorEstimates}

We consider an auxiliary problem that seeks
$(\bpsi,\xi)$ satisfying
\begin{equation}\label{dual-problem}
\begin{split}
\curl(\nu \curl\bpsi)-\nabla \xi &=\be_h\quad \mbox{in}\;\Omega,\\
\nabla\cdot\bpsi&=0\quad\mbox{in}\;\Omega,\\
\bpsi\times\bn &= 0\quad\mbox{on}\;\partial\Omega,\\
\xi &= 0\quad\mbox{on}\;\partial\Omega.\\
\end{split}
\end{equation}
Assume that the problem (\ref{dual-problem}) has the
$[H^{1+s}(\Omega)]^3\times H^s(\Omega)$-regularity property in the
sense that the solution $(\bpsi, \xi)\in [H^{1+s}(\Omega)]^3\times
H^s(\Omega)$ and the following a priori estimate holds true:
\begin{equation}\label{reg}
\|\bpsi\|_{1+s}+\|\xi\|_s\le C\|\be_h\|,
\end{equation}
where $0< s\le 1$.

\medskip
\begin{theorem} \label{L2-bd}
Let $(\bu, p)\in  [H^{t+1}(\Omega)]^3\times [H^{t}(\Omega)]$ and $(\bu_h,p_h)\in V_h\times W_h$ be the
solutions of (\ref{moment1})-(\ref{bc1}) and (\ref{wg1})-(\ref{wg2}),
respectively. Let $\frac12 <t\le k$ and $0 < s \le 1$. Then,
\begin{equation}\label{l2-error}
\|\bQ_k\bu-\bu_h\|\le Ch^{t+s}(\|\bu\|_{t+1}+\|p\|_t).
\end{equation}
\end{theorem}

\begin{proof}
Testing the first equation of (\ref{dual-problem}) by $\be_h$ gives
\[
\|\bQ_k\bu-\bu_h\|^2=(\curl(\nu\curl\bpsi),\
\be_h)-(\nabla\xi,\ \be_h).
\]
Using (\ref{m7}) and (\ref{m8}) (with $\bpsi, \xi, \be_h$ in the
place of $\bu, p, \bv$, respectively), the above equation becomes
\a{ 
\|\bQ_k\bu-\bu_h\|^2  &= (\nu\cw\bQ_k\bpsi,\ \cw\be_h)-(\be_h,\ \nabla_w(Q_h\xi))\\
  & \quad \ +E_2(\bpsi, \be_h)+E_1(\bpsi, \be_h)-E_3(\xi,\be_h)\\&= a(\bQ_k \bpsi,\ \be_h)-b(\be_h,\Q_h\xi)-s_1(\bQ_k\bpsi,\ \be_h)\\
   &\quad \ +E_2(\bpsi, \be_h)+E_1(\bpsi, \be_h)-E_3(\xi,\be_h). }
The error equation (\ref{ee2}) implies
\[
b(\be_h,\ Q_{k-1}\xi)=-s_2(\epsilon_h,\ Q_{k-1}\xi)+E_4(\bu, Q_{k-1}\xi)+s_2(Q_{k-1}p,Q_{k-1}\xi).
\]
Combining the two equations above yields
\an{\ad{
\|\bQ_k\bu-\bu_h\|^2&= a(\bQ_k \bpsi,\ \be_h)-s_2(\epsilon_h,\ Q_{k-1}\xi)
     +E_4(\bu, Q_{k-1}\xi)\\
   &\quad \ +s_2(Q_{k-1}p,Q_k\xi) 
  - s_1(\bQ_k\bpsi,\ \be_h)\\
   &\quad\ +E_2(\bpsi, \be_h)+E_1(\bpsi, \be_h)+E_3(\xi,\be_h).}\label{d0} }
It now follows from the definition of $\bQ_k$, \eqref{d-g}, the usual integration by parts, and the
second equation of (\ref{dual-problem}) that
\begin{eqnarray*}
& & b(\bQ_k\bpsi,\ \epsilon_h)\\
&=&(\bQ_k\bpsi, \nabla_w\epsilon_h)_{\T_h}=\l \{\epsilon_h\},\ \bQ_k\bpsi\cdot\bn\r_{\partial\T_h}-(\epsilon_h,\ \nabla\cdot(\bQ_k\bpsi))_{\T_h}\\
&=&\l \{\epsilon_h\}-\epsilon_h,\ \bQ_k\bpsi\cdot\bn\r_{\partial\T_h}+(\nabla\epsilon_h,\ \bpsi)_{\T_h}\\
&=&-\l [\epsilon_h],\ \bQ_k\bpsi\cdot\bn\r_{\partial\T_h}+\l [\epsilon_h],\ \bpsi\cdot\bn\r_{\partial\T_h}\\
&=&E_4(\bpsi,\epsilon_h),
\end{eqnarray*}
where we have used the fact that $\sum_{T\in\T_h}\l \{\epsilon_h\},\ \bpsi\cdot\bn\r_\pT=0$.
Adding and subtracting $b(\bQ_k\bpsi,\ \epsilon_h)$ to the equation (\ref{d0}) and using the equation above and (\ref{ee1}), we have
\an{\ad{ &\quad\
\|\bQ_k\bu-\bu_h\|^2\\
&= a(\bQ_k\bpsi,\ \be_h)-b(\bQ_k\bpsi,\ \epsilon_h)+E_4(\bpsi,\epsilon_h)\\
&\quad\ - s_2(\epsilon_h,\ Q_{k-1}\xi)+E_4(\bu, Q_{k-1}\xi)+s_2(Q_{k-1}p,Q_{k-1}\xi) \\
&\quad\ - s_1(\bQ_k\bpsi,\ \be_h)+E_2(\bpsi, \be_h)+E_1(\bpsi, \be_h)+E_3(\xi,\be_h) \\
&= -E_1(\bu,\bQ_k\bpsi)-E_2(\bu,\bQ_k\bpsi)+E_3(p,\bQ_k\bpsi) \\
 & \quad\ +s_1(\bQ_k\bu,\bQ_k\bpsi) +E_4(\bpsi,\epsilon_h) 
 - s_2(\epsilon_h,\ Q_{k-1}\xi)\\
&\quad\ +E_4(\bu, Q_{k-1}\xi)+s_2(Q_{k-1}p,Q_{k-1}\xi) - s_1(\bQ_k\bpsi,\ \be_h)\\
&\quad \ +E_2(\bpsi, \be_h)+E_1(\bpsi, \be_h)+E_3(\xi,\be_h).}\label{d1} }
We will provide the estimates for the twelve terms on the last line of (\ref{d1}). Among them, we can bound six of them by using Lemma \ref{Lemma:myestimates} with $t=s$ and (\ref{err1}),
\begin{eqnarray*}
|s_1(\bQ_k\bpsi,\ \be_h)| &\le& Ch^s\|\bpsi\|_{1+s} \3bar\be_h\3bar\le Ch^{t+s}(\|\bu\|_{t+1}+\|p\|_t) \|\bpsi\|_{1+s},\\
|s_2(\epsilon_h,\ Q_h \xi)|&\le& Ch^s\|\xi\|_{s}\ |\epsilon_h|_{0,h}\le Ch^{t+s}(\|\bu\|_{t+1}+\|p\|_t) \|\xi\|_{s},\\
|E_1(\bpsi,\bv)|&\le& Ch^s\|\bpsi\|_{1+s} \3bar\be_h\3bar\le Ch^{t+s}(\|\bu\|_{t+1}+\|p\|_t) \|\bpsi\|_{1+s},\\
|E_2(\bpsi,\be_h)|&\le& Ch^s\|\bpsi\|_{1+s} \3bar\be_h\3bar\le Ch^{t+s}(\|\bu\|_{t+1}+\|p\|_t) \|\bpsi\|_{1+s},\\
|E_3(\xi,\be_h)|&\le& C h^s\|\xi\|_{s}\3bar\be_h\3bar\le Ch^{t+s}(\|\bu\|_{t+1}+\|p\|_t) \|\xi\|_{s}, \\
|E_4(\bpsi,\epsilon_h)|&\le& Ch^s\|\bpsi\|_{1+s}|\epsilon_h|_{0,h}\le Ch^{t+s}(\|\bu\|_{t+1}+\|p\|_t) \|\bpsi\|_{1+s}.
\end{eqnarray*}
Using the trace inequality (\ref{trace}), Cauchy-Schwarz inequality, and the properties of the projection operators $\bQ_k$ and $Q_{k-1}$, we have
\begin{eqnarray*}
|s_1(\bQ_k\bu,\ \bQ_k\bpsi)|&=&|\sum_{T\in\T_h} h_T^{-1}(\langle [\bQ_k\bu]\times\bn,\; [\bQ_k\bpsi]\times\bn\rangle_\pT\\
  & & +\langle [\bQ_k\bu]\cdot\bn,\; [\bQ_k\bpsi]\cdot\bn\rangle_{\pT\setminus\partial\Omega})|\\
&=&|\sum_{T\in\T_h} h_T^{-1}(\langle [\bQ_k\bu-\bu]\times\bn,\; [\bQ_k\bpsi-\bpsi]\times\bn\rangle_\pT \\
& & +\langle [\bQ_k\bu-\bu]\cdot\bn,\; [\bQ_k\bpsi-\bpsi]\cdot\bn\rangle_{\pT\setminus\partial\Omega})|\\
&\le& Ch^{t+s}\|\bu\|_{t+1} \|\bpsi\|_{1+s},
\end{eqnarray*}
and
\begin{eqnarray*}
|s_2(Q_{k-1} p,\ Q_{k-1}\xi)|&=&|\sum_{T\in\T_h} h_T\langle [Q_{k-1} p],\; [Q_{k-1}\xi]\rangle_\pT|\\
&\le &|\sum_{T\in\T_h}h_T |\langle [Q_{k-1}p- p],\; [Q_{k-1}\xi-\xi]\rangle_\pT|\\
&\le& Ch^{t+s}\|p\|_{t} \|\xi\|_{s}.
\end{eqnarray*}
Similarly, we have  that
\begin{eqnarray*}
|E_1(\bu,\ \bQ_k\bpsi)|&=& |\l\nu\curl\bu-\bQ_k(\nu\curl\bu),\,[\bQ_k\bpsi]\times\bn\r_{\partial\T_h}|\\
&=& |\l\nu\curl\bu-\bQ_k(\nu\curl\bu),\,[\bQ_k\bpsi-\bpsi]\times\bn\r_{\partial\T_h}|\\
&\le& Ch^{t+s}\|\bu\|_{t+1} \|\bpsi\|_{1+s}.
\end{eqnarray*}
and
\begin{eqnarray*}
|E_3(p,\bQ_k\bpsi)|&=& |\l \{p-Q_{k-1}p\},\,[\bQ_k\bpsi]\cdot\bn\r_{\partial\T_h}|\\
&=& |\l \{p-Q_{k-1}p\},\,[\bQ_k\bpsi-\bpsi]\cdot\bn\r_{\partial\T_h}|\\
&\le& Ch^{t+s}\|p\|_{t} \|\bpsi\|_{1+s}.
\end{eqnarray*}
and
\begin{eqnarray*}
|E_4(\bu,Q_{k-1}\xi)|&=&|\l (\bu-\bQ_k\bu)\cdot\bn,\,[Q_{k-1}\xi]\r_{\partial\T_h}|\\
&=&|\l (\bu-\bQ_k\bu)\cdot\bn,\,[Q_{k-1}\xi-\xi]\r_{\partial\T_h}|\\
&\le& Ch^{t+s}\|\bu\|_{t+1} \|\xi\|_{s}.
\end{eqnarray*}
Using (\ref{e10}),  (\ref{p1}),      and the trace inequality (\ref{trace}), we have
\begin{eqnarray*}
|E_2(\bu,\ \bQ_k\bpsi)|&=&|(\bQ_k\curl\bu-\cw\bQ_k\bu,\ \cw\bQ_k\bpsi)|\\
&=& |\l \{\bu-\bQ_k\bu\}, [\cw\bQ_k\bpsi]\times\bn\r_{\partial\T_h}|\\
&=&| \l \{\bu-\bQ_k\bu\}, [\cw\bQ_k\bpsi-\curl\bpsi]\times\bn\r_{\partial\T_h}|\\
&\le& Ch^{t+s}\|\bu\|_{t+1} \|\bpsi\|_{1+s}.
\end{eqnarray*}
Substituting all the estimates above into (\ref{d1}) and using (\ref{reg})
imply the  error estimate (\ref{l2-error}). This completes
the proof of the theorem.
\end{proof}

\section{Numerical Results}\label{Section:NE}

We solve the model problem \eqref{moment1}--\eqref{bc1} on the unit cube domain $\Omega=(0,1)^3$,
   with $\mu=\epsilon=1$.
The exact solution is chosen as
\an{\label{s1} \b u=\p{ z^2\\ x^3\\ y^4 }, \quad p=x^4,  }
which defines the functions $\b f_1$, $g_1$, and non-homogeneous boundary conditions in
   \eqref{moment1}--\eqref{bc1}.
The computation is done on uniform cubic meshes shown in Figure \ref{grid3d}.
In Tables \ref{t1}--\ref{t2},  we list the errors and the computed order of convergence
  in various norms, by the  $P_k (k=1,2,3,4)$ weak Galerkin finite elements \cite{Mu-W-Y-Z} and
  by the new $P_k (k=1,2,3,4)$   modified weak Galerkin finite elements.
 The optimal-order convergence is achieved in all cases.  
For the $P_1$ finite element methods,  the number of unknowns in the modified weak Galerkin
  finite element discrete equations is about one-third of that of the weak Galerkin equations.
But both methods produce solutions of about equal accuracy, in all the tests.

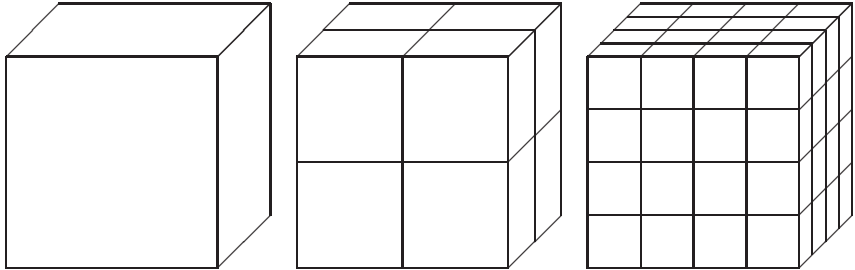
\begin{figure}[ht]
\begin{center}
 \setlength\unitlength{1pt}
    \begin{picture}(320,118)(0,3)
    \put(0,0){\begin{picture}(110,110)(0,0)
       \multiput(0,0)(80,0){2}{\line(0,1){80}}  \multiput(0,0)(0,80){2}{\line(1,0){80}}
       \multiput(0,80)(80,0){2}{\line(1,1){20}} \multiput(0,80)(20,20){2}{\line(1,0){80}}
       \multiput(80,0)(0,80){2}{\line(1,1){20}}  \multiput(80,0)(20,20){2}{\line(0,1){80}}
      \end{picture}}
    \put(110,0){\begin{picture}(110,110)(0,0)
       \multiput(0,0)(40,0){3}{\line(0,1){80}}  \multiput(0,0)(0,40){3}{\line(1,0){80}}
       \multiput(0,80)(40,0){3}{\line(1,1){20}} \multiput(0,80)(10,10){3}{\line(1,0){80}}
       \multiput(80,0)(0,40){3}{\line(1,1){20}}  \multiput(80,0)(10,10){3}{\line(0,1){80}}
      \end{picture}}
    \put(220,0){\begin{picture}(110,110)(0,0)
       \multiput(0,0)(20,0){5}{\line(0,1){80}}  \multiput(0,0)(0,20){5}{\line(1,0){80}}
       \multiput(0,80)(20,0){5}{\line(1,1){20}} \multiput(0,80)(5,5){5}{\line(1,0){80}}
       \multiput(80,0)(0,20){5}{\line(1,1){20}}  \multiput(80,0)(5,5){5}{\line(0,1){80}}
      \end{picture}}

    \end{picture}
    \end{center}
\caption{ The first three grids for the computation in Tables \ref{t1}--\ref{t4}.  }
\label{grid3d}
\end{figure}

\begin{table}[ht]
  \centering  \renewcommand{\arraystretch}{1.2}
  \caption{Error profile on the meshes shown as in Figure \ref{grid3d}
    for the solution \eqref{s1}, by the $P_1$ elements. }
  \label{t1}
\begin{tabular}{c|cc|cc|cc|r}
\hline
grid & \multicolumn{2}{c|}{ $\|Q_k \bm{u}-\bm{u}_h\| $    $O(h^r)$}
   & \multicolumn{2}{c|}{  $\3bar{Q_k \bm{u}-\bm{u}_h}\3bar$  $O(h^r)$}
    &  \multicolumn{2}{c|}{ $ \| p-p_h\| $  $O(h^r)$}& $\dim$  
  \\ \hline
    &  \multicolumn{7}{c}{ By the  $P_1$ weak Galerkin finite element method \cite{Mu-W-Y-Z}. }   \\
\hline 
 1&   0.124E+01& 0.0&   0.267E+01& 0.0&   0.114E+00& 0.0&     31 \\
 2&   0.372E+00& 1.7&   0.156E+01& 0.8&   0.122E+00& 0.0&    296 \\
 3&   0.107E+00& 1.8&   0.889E+00& 0.8&   0.817E-01& 0.6&   2560 \\
 4&   0.315E-01& 1.8&   0.487E+00& 0.9&   0.423E-01& 0.9&  21248 \\ 
 5&   0.879E-02& 1.9&   0.253E+00& 0.9&   0.203E-01& 1.1& 173056 \\
\hline
    &  \multicolumn{7}{c}{ By the  $P_1$ modified weak Galerkin finite element method. }   \\
\hline
 1&   0.124E+01& 0.0&   0.267E+01& 0.0&   0.114E+00& 0.0&     31 \\
 2&   0.294E+00& 2.1&   0.134E+01& 1.0&   0.120E+00& 0.0&    176 \\
 3&   0.757E-01& 2.0&   0.598E+00& 1.2&   0.817E-01& 0.6&   1120 \\
 4&   0.166E-01& 2.2&   0.236E+00& 1.3&   0.452E-01& 0.9&   7808 \\ 
 5&   0.387E-02& 2.1&   0.933E-01& 1.3&   0.219E-01& 1.0&  57856 \\
\hline
    \end{tabular}%
\end{table}%

\begin{table}[ht]
  \centering  \renewcommand{\arraystretch}{1.2}
  \caption{Error profile on the meshes shown as in Figure \ref{grid3d}
    for the solution \eqref{s1}, by the $P_2$ elements. }
  \label{t2}
\begin{tabular}{c|cc|cc|cc|r}
\hline
grid & \multicolumn{2}{c|}{ $\|Q_k \bm{u}-\bm{u}_h\| $    $O(h^r)$}
   & \multicolumn{2}{c|}{  $\3bar{Q_k \bm{u}-\bm{u}_h}\3bar$  $O(h^r)$}
    &  \multicolumn{2}{c|}{ $ \| p-p_h\| $  $O(h^r)$}& $\dim$  
  \\ \hline
    &  \multicolumn{7}{c}{ By the  $P_2$ weak Galerkin finite element method \cite{Mu-W-Y-Z}. }   \\
\hline 
 1&   0.475E+00& 0.0&   0.131E+01& 0.0&   0.160E+00& 0.0&     70 \\
 2&   0.738E-01& 2.7&   0.372E+00& 1.8&   0.783E-01& 1.0&    668 \\
 3&   0.105E-01& 2.8&   0.980E-01& 1.9&   0.241E-01& 1.7&   5776 \\
\hline
    &  \multicolumn{7}{c}{ By the  $P_2$ modified weak Galerkin finite element method. }   \\
\hline
 1&   0.475E+00& 0.0&   0.131E+01& 0.0&   0.160E+00& 0.0&     70 \\
 2&   0.695E-01& 2.8&   0.359E+00& 1.9&   0.798E-01& 1.0&    416 \\
 3&   0.127E-01& 2.4&   0.106E+00& 1.8&   0.243E-01& 1.7&   2752 \\
\hline
    \end{tabular}%
\end{table}%

\begin{table}[ht]
  \centering  \renewcommand{\arraystretch}{1.2}
  \caption{Error profile on the meshes shown as in Figure \ref{grid3d}
    for the solution \eqref{s1}, by the $P_3$ elements. }
  \label{t3}
\begin{tabular}{c|cc|cc|cc|r}
\hline
grid & \multicolumn{2}{c|}{ $\|Q_k \bm{u}-\bm{u}_h\| $    $O(h^r)$}
   & \multicolumn{2}{c|}{  $\3bar{Q_k \bm{u}-\bm{u}_h}\3bar$  $O(h^r)$}
    &  \multicolumn{2}{c|}{ $ \| p-p_h\| $  $O(h^r)$}& $\dim$  
  \\ \hline
    &  \multicolumn{7}{c}{ By the  $P_3$ weak Galerkin finite element method \cite{Mu-W-Y-Z}. }   \\
\hline 
 1&   0.138E+00& 0.0&   0.344E+00& 0.0&   0.114E+00& 0.0&    130 \\
 2&   0.101E-01& 3.8&   0.432E-01& 3.0&   0.193E-01& 2.6&   1232 \\
 3&   0.674E-03& 3.9&   0.538E-02& 3.0&   0.270E-02& 2.8&  10624 \\
\hline
    &  \multicolumn{7}{c}{ By the  $P_3$ modified weak Galerkin finite element method. }   \\
\hline
 1&   0.138E+00& 0.0&   0.344E+00& 0.0&   0.114E+00& 0.0&    130 \\
 2&   0.930E-02& 3.9&   0.340E-01& 3.3&   0.195E-01& 2.5&    800 \\
 3&   0.554E-03& 4.1&   0.301E-02& 3.5&   0.273E-02& 2.8&   5440 \\
\hline
    \end{tabular}%
\end{table}%

\begin{table}[ht]
  \centering  \renewcommand{\arraystretch}{1.2}
  \caption{Error profile on the meshes shown as in Figure \ref{grid3d}
    for the solution \eqref{s1}, by the $P_4$ elements. }
  \label{t4}
\begin{tabular}{c|cc|cc|cc|r}
\hline
grid & \multicolumn{2}{c|}{ $\|Q_k \bm{u}-\bm{u}_h\| $    $O(h^r)$}
   & \multicolumn{2}{c|}{  $\3bar{Q_k \bm{u}-\bm{u}_h}\3bar$  $O(h^r)$}
    &  \multicolumn{2}{c|}{ $ \| p-p_h\| $  $O(h^r)$}& $\dim$  
  \\ \hline
    &  \multicolumn{7}{c}{ By the  $P_4$ weak Galerkin finite element method \cite{Mu-W-Y-Z}. }   \\
\hline 
 1&   0.317E-01& 0.0&   0.652E-01& 0.0&   0.319E-01& 0.0&    215 \\
 2&   0.998E-03& 5.0&   0.367E-02& 4.1&   0.222E-02& 3.8&   2020 \\
 3&   0.330E-04& 4.9&   0.209E-03& 4.1&   0.144E-03& 3.9&  17360 \\
\hline
    &  \multicolumn{7}{c}{ By the  $P_4$ modified weak Galerkin finite element method. }   \\
\hline
 1&   0.317E-01& 0.0&   0.652E-01& 0.0&   0.319E-01& 0.0&    215 \\
 2&   0.892E-03& 5.2&   0.286E-02& 4.5&   0.227E-02& 3.8&   1360 \\
 3&   0.295E-04& 4.9&   0.165E-03& 4.1&   0.147E-03& 3.9&   9440 \\
\hline
    \end{tabular}%
\end{table}%

\end{document}